\documentclass[letterpaper,10pt,twocolumn]{article}
\usepackage{geometry}
\usepackage[utf8]{inputenc}
\usepackage{epsfig}
\usepackage{epstopdf}
\usepackage{graphicx,latexsym,amsmath,amssymb,enumerate,cases}
\usepackage{amsthm}
\usepackage{amscd}
\usepackage{array}
\usepackage{textcomp}
\usepackage{float}
\usepackage{bbding}
\usepackage{amsfonts}
\usepackage{makecell}
\usepackage{mathrsfs}
\usepackage{tikz}
\usepackage{amssymb}
\usepackage{booktabs}
\usepackage{bm}
\usepackage{diagbox}
\usepackage{url}
\usepackage{verbatim}
\usepackage{romannum}

\newtheorem{theorem}{Theorem}[section]
\newtheorem{lemma}{Lemma}[section]

\newtheorem{definition}{Definition}[section]
\newtheorem{remark}{Remark}[section]

\newtheorem{assumption}{Assumption}[section]
\newtheorem{notation}{Notation}[section]
\newtheorem{problem}{Problem}
\allowdisplaybreaks  
\begin{document}

\title{{A Slow-Fast Stochastic Framework for Zeroth-Order Distributed Time-Varying Optimization}}
\author{{Wanying Li\thanks{E-mail:liwanying\_scu@163.com}, Nanjing Huang\thanks{Corresponding author,  E-mail:\;nanjinghuang@hotmail.com; njhuang@scu.edu.cn}} \\
{\small\it Department of Mathematics, Sichuan University, Chengdu, Sichuan 610064, P.R. China}}
\date{}
\maketitle
{\bf Abstract}.
This paper investigates the distributed time-varying optimization of stochastic multi-agent systems (SMASs) using only zero-order information. Unlike existing methods that directly couple gradient estimation and optimization updates on a single time scale, this paper constructs a novel stochastic singular perturbation framework by introducing auxiliary fast systems. The proposed scheme naturally forms a slow-fast coupling structure: by introducing auxiliary variables and constructing fast subsystems to generate smooth gradient estimates, while the agent's state evolution, as the slow subsystem, performs distributed optimization and consensus.
The convergence of the proposed scheme is analyzed using stochastic singular perturbation techniques and stochastic Lyapunov theory. The results show that the fast subsystem converges rapidly to the instantaneous stochastic gradient estimates, while the slow subsystem achieves practically fixed-time consensus (Pfxc) in probability and asymptotically bounded tracks the time-varying optimal trajectory. Furthermore, this paper establishes explicit bounds to characterize the effects of parameters, stochastic disturbances, and the properties of the objective function on tracking performance.
Finally, the theoretical results are validated through numerical simulations.


\footnotetext{This work was developed by the National Natural Science Foundation of China (12171339, 12471296).}

\section{Introduction}

Distributed time-varying optimization has attracted increasing attention due to its capability of enabling networked agents to cooperatively track time-varying optimal trajectories. In contrast to static optimization, where the objective functions does not change over time, time-varying optimization considers scenarios in which the local cost functions continuously evolve with time, leading to a time-varying optimal trajectories that must be tracked in real time. To address this challenge, various distributed gradient-based algorithms have been developed in recent years. Representative approaches include consensus-based estimation methods \cite{2017-Rahili,2020-Huang,2024-Cui,2025-Zhou}, sliding-mode control-based method \cite{2017-Sun,2022-Sun}, primal-dual dynamics methods \cite{2022-Pedro}, prediction-correction methods \cite{2019-Simonetto}, and exosystem dynamics methods \cite{2024-Ding}.
Despite these advances, most existing distributed time-varying optimization methods require at least first-order exact gradient information. Unfortunately, in many practical cases, obtaining such information can be costly or even impossible, leaving only the function values available. Therefore, gradient-based methods are difficult or even impossible to implement in such cases.

To overcome these limitations, distributed zero-order optimization methods have been extensively studied. Existing distributed zero-order optimization algorithms construct gradient estimates using stochastic perturbations, Gaussian smoothing techniques, and finite-difference estimators, and are analyzed using regret-based performance metrics \cite{2017-Nesterov,2019-Hajinezhad,2022-Yi}. Recently, several zero-order algorithms have been extended to solve the stochastic optimization problems \cite{2022-Yi,2023-Kornilov,2024-Guy} and online optimization problems \cite{2024-Xu,2024-Cao,2025-Hou}. Under appropriate conditions, regret-based performance analysis can still be established. However, most existing distributed zero-order optimization algorithms have been developed in discrete-time frameworks. While such modeling approaches are suitable for numerical implementation, they fail to explicitly capture the continuous evolution of physical processes in many engineering applications, such as multi-robot coordination \cite{2024-Hu}, UAV swarms \cite{2025-Liu}, and sensor networks \cite{2004-Rabbat,2017-Lu}, in which the optimization process and state evolution occur simultaneously in continuous time. Thus, it would be important and interesting to investigate the continuous-time optimization frameworks, as they reveal the inner dynamics of optimization algorithms and allow for the application of methods from dynamical systems and control theory.

Meanwhile, environmental disturbances, uncertainties, communication noises, and random external perturbations are everywhere in multi-agent systems. To model uncertainty in multi-agent systems, stochastic differential equations have been widely adopted, driving a significant amount of research in SMASs over the past decade. Various consensus and stability results have been established using stochastic Lyapunov techniques and stochastic stability theory \cite{2019-You,2024-Tang,2025-Li}.

When dealing with time-varying objective functions, stochasticity, and zeroth-order information simultaneously, several fundamental challenges arise. First, the global optimal trajectory continuously evolves over time, requiring the algorithms to track a moving target rather than converge to a fixed point. Second, since accurate gradient information is not available, the reconstruction must be performed online using the function values. Third, stochasticity influence the evolution of agent states and inevitably affect the tracking performance of the algorithms. These coupled challenges greatly increase the complexity of algorithm design and theoretical analysis.

To address these challenges, this paper proposes a slow-fast stochastic framework for zero-order distributed time-varying optimization. Specifically, we introduce an auxiliary variable and construct a fast subsystem to generate a smooth gradient estimate for each agent $i$, while the state evolution, as slow subsystem, is designed to handle distributed optimization and consensus; the resulting system exhibits a two-timescale structure.
The results demonstrate that, despite the presence of stochastic perturbations and gradient estimation errors, the proposed system enables the SMASs to achieve Pfxc in probability and perform bounded tracking of the time-varying global optimal trajectory. The main contributions of this paper are summarized as follows:
\begin{itemize}
\item[$\bullet$] We formulate a distributed zeroth-order optimization problem with time-varying objective functions, which unifies time-varying optimization, gradient-free information structures, and SMASs within a common framework.

\item[$\bullet$] We design a novel slow-fast coupled system by introducing an auxiliary smooth gradient estimator as the fast subsystem. This auxiliary subsystem is capable of rapidly tracking instantaneous stochastic gradient estimator, providing smooth gradient estimator for the agent's state evolution (slow system), forms a two-timescale optimization framework.

\item[$\bullet$] We propose a stochastic singular perturbation analysis framework for the proposed slow-fast system. The study shows that the auxiliary smooth gradient estimator subsystem is capable of achieving exponential tracking of instantaneous stochastic gradient estimator, and under limit conditions, an approximate relationship between the slow system and a reduced system is obtained.

\item[$\bullet$] We establish the Pfxc in probability and bounded tracking of the time-varying global optimal trajectory in the presence of gradient estimation error, stochastic disturbances, and consensus constraints. We also explicitly demonstrate the influence of design parameters on consensus and tracking performance.
\end{itemize}

The remainder of this paper is structured as follows. The next section reviews some preliminaries. Subsequently, in Section 3, we present the formulation of the zero-order distributed time-varying optimization problem and describe the slow-fast system we designed. In Section 4, we present our main results and their proofs. Section 5 presents the simulation experiment to verify the validity of the theoretical results. Finally, Section 6 summarizes the conclusions of this paper.

\section{Preliminaries}
\subsection{Basic Lemmas}
Below, we will introduce several lemmas to lay the foundation for our subsequent discussion.

\begin{lemma}\cite{2018-Still}\label{lem:optimal}
Suppose $f(t,x)$ is a $\mathcal{C}^2$-function satisfying $\nabla f(\bar{t},\bar{x})=0$ and $\nabla^2 f(\bar{t},\bar{x})\succ 0$. If $\bar{x}$ is an isolated strict local minimizer of $f(t,x)$ at time $\bar{t} \in \mathbb{R}_+$, then there exists a $C^1$-function $x:B_{\varepsilon}(\bar{t})\rightarrow \mathbb{R}^n$ where $B_{\varepsilon}(\bar{t})$ is a neighborhood of $\bar{t}$ $(\varepsilon>0)$, such that $x(\bar{t})=\bar{x}$ and $x(t)$ is an isolated strict local minimizer of $f(t,x)$ for any $t\in B_{\varepsilon}(\bar{t})$.
\end{lemma}

\begin{lemma}\cite{1952-Hardy}\label{lem:inequality}
The following inequalities hold:
\begin{enumerate}
\item{}for $x \in \mathbb{R}^d$, $\|x\|_q \leq \|x\|_p \leq d^{{\frac{1}{p}}-\frac{1}{q}}\|x\|_q$, where $0<p<q$;
\item{}for $x_i\in \mathbb{R}_+$, one has
\begin{align*}
&\left(\sum_{i=1}^nx_{i}\right)^{k} \leq \sum_{i=1}^nx_{i}^{k}\leq n^{1-k}\left(\sum_{i=1}^n x_{i}\right)^{k}\;(0<k<1),\\
&\sum_{i=1}^n x_{i}^{k}\leq \left(\sum_{i=1}^n x_{i}\right)^{k}\leq n^{k-1}\sum_{i=1}^nx_{i}^{k}\;(k\geq1).
\end{align*}
\end{enumerate}
\end{lemma}

In addition, in research of multi-agent systems, we often use the signum-type function: for $x=[x_1,x_2,\cdots,x_d]^T\in \mathbb{R}^d$ and $m>0$, $sig^m(x)=[sig^m(x_{1}), sig^m(x_{2}), \cdots , sig^m(x_{d})]^{T}$, where $sig^m(x_{i})=sign(x_{i})|x_{i}|^m$ and
$sign(x_{i}) =
\begin{cases}
  1, & \mbox{if $x_{i}>0$}; \\
  -1, & \mbox{if $x_{i}<0$}; \\
  0, & \mbox{if $x_{i}=0$}.
\end{cases}
$
\begin{lemma}\label{lem:sig}
The function $sig^m(x)$ with $m>0$ is monotonically non-decreasing, i.e.,
$$
(x-y)^T\left(sig^m(x)-sig^m(y)\right)\geq0, \,\forall x,y.
$$
\end{lemma}
The Table \ref{Notations} lists the notations that may appear in our discussion.
\begin{table}[h]
  \centering
  \caption{}
  \label{Notations}
  \begin{tabular*}{\linewidth}{@{}l@{\extracolsep{\fill}}p{5.8cm}@{}}
  \hline
  Notation &Meaning  \\
  \hline
  $\nabla f(t,x)$   &  the gradient of $f(t,x)$ at $x$
  \\
  \hline
  $\lambda_{i}[\mathcal{L}]$  & The $i$-th eigenvalue after ranking the eigenvalues of matrix $\mathcal{L}$ from smallest to largest 
  \\
  \hline
  $\|x\|_p$  & the $p$-norm of vector $x\in \mathbb{R}^n$, $\|x\|_{p}=(\sum_{i=1}^{m}|x_{i}|^{p})^{\frac{1}{p}}$. When $p = 2$, it is denoted as $\|x\|$. \\
  \hline
  $\|A\|_F$  & the Frobenius-norm of $A\in \mathbb{R}^{n\times n}$, $\|A\|_F=\sqrt{\sum_{i,j=1}^{n}a_{ij}^2}=\sqrt{\text{trace}[A^TA]}$\\
  \hline
  $(\Omega, \mathcal{F}, P)$ & a complete probability space, where $\Omega$ is a sample space, $\mathcal{F}$ is a $\sigma$-field and $P$ is a probability measure.\\
  \hline
  \end{tabular*}
\end{table}

\subsection{Graph theory}
Suppose $\mathcal{G}$ is an interaction graph consisting of $n$ agents labeled by $\mathcal{N}=\{1,2,\cdots,n\}$. Two matrices play a crucial role in our subsequent discussion. The adjacency matrix $\mathcal{A} = [a_{ij}]_{n\times n}$ describes the information exchange. The graph $\mathcal{G}$ is undirected when $a_{ij} = a_{ji}$, and directed otherwise. For undirected cases, $a_{ij} > 0$ if there is an edge between agent $i$ and agent $j$, and $a_{ij} = 0$ otherwise, with $a_{ii} = 0$ assumed generally. Furthermore, the Laplacian matrix $\mathcal{L} = [l_{ij}]_{n\times n}$ is also important, defined by $l_{ij}=-a_{ij}$ for $i \neq j$ and $l_{ii} = \sum_{j=1}^{n}a_{ij}$.

\begin{lemma}\cite{2001-Godsil}\label{lem:laplacian}
  For an undirected connected graph $\mathcal{G}$, its Laplacian matrix $\mathcal{L}$ is semipositive definite and has n non-negative eigenvalues $0=\lambda_1\leq\lambda_2\leq\dots\leq\lambda_n$. Moreover, $\lambda_2$ and $\lambda_n$ satisfy
  \begin{equation*}
    \lambda_{2}[\mathcal{L}]=\min_{v\neq0, \textbf{1}^Tv=0}\frac{v^T\mathcal{L}v}{v^Tv}, \quad \lambda_{n}[\mathcal{L}]=\max_{v\neq0, \textbf{1}^Tv=0}\frac{v^T\mathcal{L}v}{v^Tv}.
  \end{equation*}
\end{lemma}

\subsection{Stability}
Let $B(t)$ be an $d$-dimensional Brownian motion defined on the probability space $(\Omega, \mathcal{F}, P)$. We consider the stochastic differential equations (SDE) \cite{book-SDE}:
\begin{equation}\label{SDE}
    dx_t = b(t, x_t)dt + \sigma(t, x_t)dB(t), \quad x_0 = x^o,
\end{equation}
where $x_t \in \mathbb{R}^d$ denotes the system state at time $t$. The initial state $x^o$ is an $\mathbb{R}^d$-valued random variable with finite second moment, i.e., $\mathbb{E}(\|x^o\|^2) < \infty$. In this model, $b(\cdot, \cdot):\mathbb{R}_+ \times \mathbb{R}^d\rightarrow\mathbb{R}^d$ and $\sigma(\cdot, \cdot):\mathbb{R}_+ \times \mathbb{R}^d\rightarrow\mathbb{R}^{d \times d}$ correspond to the drift and diffusion terms, respectively. In the following text, $\sigma(\cdot, \cdot)$ is often abbreviated as $\sigma$.

For each $ V \in C_1^2(\mathbb{R}_+\times \mathbb{R}^d; \mathbb{R}_+)$ and the system \eqref{SDE}, the infinitesimal generator defined by $\mathcal{L}V(t, x(t)): \mathbb{R}_+ \times \mathbb{R}^d \to \mathbb{R}$ \cite{book-SDE},
\begin{eqnarray*}
\mathcal{L}V(t, x) &=& \frac{\partial V(t, x)}{\partial t} + \frac{\partial V(t, x)}{\partial x} b(t, x) \\
&&+ \frac{1}{2} \text{trace} \left[ \sigma^T(t, x) \frac{\partial^2 V(t, x)}{\partial x^2} \sigma(t, x) \right].
\end{eqnarray*}

\begin{lemma} \cite{2015-Xu,2020-He}\label{lem:GEUB}
Let $V \in C_1^2(\mathbb{R}_+\times \mathbb{R}^d; \mathbb{R}_+)$, if there are constants $c_1>0$, $c_2>0$, $k_1>0$ and $k_2\geq0$ such that $c_1\|x\|^p\leq V(t,x)\leq c_2\|x\|^p (p>0)$ and
\begin{equation*}
\mathcal{L}V(t, x(t)) \leq -k_1 V(t, x(t)) + k_2.
\end{equation*}
Then system \eqref{SDE} is globally $p$-th moment exponentially ultimately bounded, i.e., $ \forall t\geq t_0$,
\[\mathbb{E}[V(t, x(t))]\leq V(t_0, x_0) e^{-k_1 (t-t_0)}+\frac{k_2}{k_1}(1-e^{-k_1 (t-t_0)}).\]
Moreover, for $p=2$, this is referred to as globally exponentially ultimately bounded in a mean-square sense (MS-GEUB).
\end{lemma}

\begin{definition}\cite{2023-min}\label{def:Sfx}
  The solution $x(t;x^o)$ of the system \eqref{SDE} is said to be practically fixed-time stable (Pfxs) in probability, if
\begin{enumerate}
\item{the system \eqref{SDE} has a unique solution for all initial states $ x^o\in \mathbb{R}^d$;}
\item{for every initial states $x^o\in \mathbb{R}^d\setminus \{0\}$, the settling time $\tau=\inf\{t: x(t;x^o)\in \{0\}\}$ is the first time $\{0\}$ is reached and it satisfies $\mathbb{E}(\tau)\leq T$. In particular, $T$ does not depend on the initial value;}
\item{for a constant $\epsilon_0>0$, $\mathbb{E}(\|x(t;x^o)\|)\leq\epsilon_0$ holds for all $t\geq\tau$.}
\end{enumerate}
\end{definition}

\begin{lemma}\cite{2023-min}\label{lem:Pfxs}
  For the system \eqref{SDE}, suppose that $b(t,x)$, $\sigma(t,x)$ satisfy the monotone condition and are locally Lipschitz continuous in $x$. If there exists a continuous radially unbounded $V\in\mathcal{C}^2$ and constants $k_1,k_2,k_3\in\mathbb{R}_
  {++},\,p\in(0,1),\,q\in(1,+\infty)$ such that
  \[\mathcal{L}V(x)\leq -k_1V(x)^{p}-k_2V(x)^{q}+k_3, \, \forall x\in\mathbb{R}^d,\]
  then the solution of system \eqref{SDE} is Pfxs in probability and the stochastic settling time $\tau = \inf \{t:\mathbb{E}(V(x))\leq \triangle\}$ satisfies
  \[\mathbb{E}(\tau)\leq \frac{(\bar{k}_1/\bar{k}_2)^{\frac{1-p}{q-p}}}{\bar{k}_1(1-p)}+ \frac{(\bar{k}_1/\bar{k}_2)^{\frac{1-q}{q-p}}}{\bar{k}_2(q-1)}=T_{\max},\, \forall x^o\in \mathbb{R}^d\backslash\{0\},\]
  where $\bar{k}_1=k_1-\frac{k_3}{\triangle^p}>0$, $\bar{k}_2=k_2-\frac{k_3}{\triangle^q}>0$ and $\triangle=\min\left\{\left(\frac{k_3}{(1-\omega)k_1}\right)^{\frac{1}{p}},\left(\frac{k_3}{(1-\omega)k_2}\right)^{\frac{1}{q}}\right\}$, and $0<\omega<1$ is a design constant.
\end{lemma}

\section{Problem Formulation}
\subsection{Distributed optimization}

\begin{problem}(Distributed optimization problem with Consensus Constraint)\label{P}
\begin{equation}\label{problem}
\left\{
\begin{array}{l}
\min \mathbb{E}\left[\frac{1}{n}\sum_{i\in n}f_i(t,x_{i})\right]\\
\mbox{s.t. }d x_{i} = u_i dt + \sigma^x_idB^x(t), \, x_{i}(0) = x_{i}^o,\\
\quad\;\; \mathbb{E}\left(\frac{1}{n}\sum_{i\in \mathcal{N}}\left\|x_{i}-\frac{1}{n}\sum_{j\in \mathcal{N}}x_{j}\right\|^2\right)\leq\varrho,
\end{array}
\right.
\end{equation}
where the dynamics of each agent $i$ follow the above stochastic differential equations, $x_{i}$, $\sigma_i^x$ and $x_{i0}$ satisfy the same assumptions as in \eqref{SDE}. Here, $f_i:\mathbb{R}_+ \times \mathbb{R}^d \rightarrow \mathbb{R}$ represents the time-varying local cost function of agent $i$ and $f_i \in C^2(\mathbb{R} \times \mathbb{R}^d)$, $u_i$ is the protocol to be designed. The scalar $\varrho >0$ quantifies an allowable deviation from consensus in mean square.
\end{problem}
\begin{remark}
Unlike common consensus constraints, the consensus constraint we proposed $$\mathbb{E}\left(\frac{1}{n}\sum_{i\in \mathcal{N}}\left\|x_{i}-\frac{1}{n}\sum_{j\in \mathcal{N}}x_{j}\right\|^2\right)\leq\varrho$$ constitutes a relaxed consensus requirement---ensuring agents remain collectively close without enforcing exact agreement, i.e. biased consensus. This change is natural in a stochastic setting, since the agent's states are described by SDEs. The presence of Brownian motion makes it impossible to maintain strict consensus in the long run.
\end{remark}

The interaction graph of the SMASs satisfies the following assumption.

\begin{assumption}\label{A-graph}
The graph $\mathcal{G}$ is undirected, static and connected with adjacency matrix $\mathcal{A} = [a_{ij}]_{n\times n}$ and Laplacian matrix $\mathcal{L}$.
\end{assumption}

\begin{notation}
In the discussion below, let $\mathcal{L}_p$ and $\mathcal{L}_q$ be the Laplacian matrix of $\mathcal{A}_p = [a_{ij}^\frac{2}{p+1}]_{n\times n}$ and $\mathcal{A}_q = [a_{ij}^\frac{2}{q+1}]_{n\times n}$ respectively.
\end{notation}

\begin{remark}
Although our discussion assumes undirected graphs, the proposed method is also effective for a class of strongly connected directed graphs with detail-balanced in weights, since these specific directed graphs can be associated with undirected graphs through appropriate processing \cite{2023-Yu,2025-Li}.
\end{remark}

\subsection{Slow-fast coupling system}
Based on the SMASs in \eqref{problem}, we propose the slow-fast coupling system (recorded as SF-SMASs) to solve  Problem \ref{P} having the following form:
\begin{eqnarray}
d x_{i} &=& u_i\, d t
+ \sigma_i^x \, dB^x(t), \label{slow}\\
\varepsilon dg_{i} &=& \kappa \left[ A_i(t,x_{i}) - r_i(t) g_{i} \right] d t+ \beta_{\varepsilon}\sqrt{\varepsilon}\sigma_i^g \, d B^g(t)\label{fast}
\end{eqnarray}
with
\begin{eqnarray*}
u_i&=&-\alpha_1 \, g_{i}
 -\alpha_2\sum_{j=1}^{n} a_{ij} (x_{i} - x_{j})\nonumber\\
&&-\alpha_3\sum_{j=1}^{n} a_{ij} sig^p(x_{i} - x_{j})
-\alpha_4\sum_{j=1}^{n} a_{ij} sig^q(x_{i} - x_{j}),
\end{eqnarray*}
where $x_{i}$ is the ``slow'' variable and $g_{i}$ is the ``fast''variable, $\epsilon\in (0,1]$ is a small parameter, $\alpha_1,\alpha_2,\alpha_3,\alpha_4,\kappa>0$, $0<p<1$ and $q>1$; $\sigma^x_i$ and $\sigma^g_i$ are abbreviations for $\sigma^x_i(t,x_{i})$ and $\sigma^g_i(t,g_{i})$ respectively, and the relevant assumptions regarding the diffusion term will be provided later; $B^x(t)$ and $B^g(t)$ are independent $d$-dimensional Brownian motions, and $A_i(t,x_{i})$ is the instantaneous stochastic gradient estimator for coordinate-wise perturbations given by
\begin{align}\label{Ai}
A_i(t,x_{i}) = \sum_{\ell=1}^{d} \frac{f_i(t,x_i + \gamma_t r_i e_\ell) - f_i(t,x_i - \gamma_t r_i e_\ell)}{2\gamma_t} \, e_\ell.
\end{align}
Here $e_\ell$ is the standard basis vector, $\gamma_t$ is a estimating parameter ($0<\gamma_t\le 1$ for all $t\ge 0$ with $\lim_{t\rightarrow \infty}\gamma_t=0$), and $r_i(t)$ is the piecewise-constant random excitation, i.e.
\[
r_i(t) = r_i^k, \quad t \in [k T_r, (k+1) T_r)\,\, \mbox{with} \,\, T_r \gg \varepsilon,
\]
where $r_i^k \sim \mbox{Unif} [1-a,\, 1+a]$ and $0<a<1$.

\begin{remark}
The fast system \eqref{fast} tracks and smooths the instantaneous stochastic gradient estimator $A_i(t,x_i)$ in real time under random perturbations, acting as an adaptive low-pass filter and providing a robust and reliable optimization direction for the slow system \eqref{slow}.
\end{remark}

\begin{remark}
The design of \eqref{Ai} is motivated by the two-point gradient estimation method \cite{2025-Qin,2022-Yuan}. Coordinate-wise gradient estimation enables high-precision gradient estimation with low bias and offers greater robustness.
\end{remark}

\section{Main results}
In this section, we will present our main results in details along with their proofs. To this end, we need to propose some assumptions as follows.

\begin{assumption}\label{A-sigma}
 The diffusion term of SF-SMASs \eqref{slow}-\eqref{fast} satisfies:
 \begin{enumerate}
\item[(i)] $\min\left\{\|\sigma^x_i\|_F, \|\sigma^g_i\|_F\right\}\leq \bar{\sigma}$ for all $x\in\mathbb{R}^d$, $t>0$, $i\in \mathcal{N}$.

\item[(ii)] $\beta_\varepsilon = o\!\left(\dfrac{1}{\sqrt{|\ln\varepsilon|}}\right)$ as $\varepsilon \to 0$.
\end{enumerate}
\end{assumption}
\begin{remark}
The diffusion coefficient of the auxiliary fast subsystem \eqref{fast} is chosen as $\beta_\varepsilon$, which is a common assumption in singular perturbation theory \cite{book-Tikhonov}.  This ensures that the stochastic perturbation captures practical uncertainties without altering the singular perturbation limit or the quasi-steady-state behavior of the fast subsystem.
\end{remark}
\begin{assumption}\label{A-L1}
For all $x,y\in \mathbb{R}^d$, $t\geq 0$, $i\in \mathcal{N}$, the objective function $f_i(t,x)$ satisfies:
 \begin{enumerate}
 \item[(i)]$f_i(t,x)$ is $\mu$-strongly convex on $x$;

\item[(ii)] The gradient of $f_i$ is Lipschitz continuous, i.e.
$\left\|\nabla f_i(t,x) - \nabla f_i(t,y)\right\|\leq L_1\|x-y\|$;

\item[(iii)]For all $j\in \mathcal{N}$, $\left\|\nabla f_i(t,x) - \nabla f_j(t,y)\right\|\leq L_2\|x-y\|+L_3$.
\end{enumerate}
\end{assumption}

\begin{assumption}\label{A-x*}
  There exists a unique optimal trajectory $x^*_t$ satisfies $\|\dot{x}^*_t\|\leq M$ ($M>0$).
\end{assumption}

\begin{remark}
The existence and uniqueness of $x^*_t$ can be guaranteed by strong convexity of $f_i$ and Lemma \ref{lem:optimal}. Assumption \ref{A-x*} is intended to ensure that the optimal trajectory does not change too quickly, thereby preventing the analysis of the slow-fast system from becoming invalid.
\end{remark}

\subsection{Stochastic Singular Perturbation Analysis}
In this subsection, we will perform a stochastic singular perturbation analysis for SF-SMASs \eqref{slow}-\eqref{fast}.
The following lemma provides a bias analysis of the instantaneous stochastic gradient estimator \eqref{Ai}.

\begin{lemma}\label{lem:Ai}
Under Assumption \ref{A-L1}, for any $i\in \mathcal{N}$, we have
\[
\frac{A_i(t,x_i) }{r_i(t)}=\nabla f_i(t,x_i)+\theta_i(t),
\]
where $\|\theta_i(t)\|\leq C_A \gamma_t$ with $C_A=\frac{1}{2}\sqrt{d}L_1(1+a)$.
\end{lemma}
\begin{proof}
For the sake of brevity, let $h=\gamma_t r_i$. Then
\begin{eqnarray}
\frac{A_i(t,x_i)}{r_i(t)} = \sum_{\ell=1}^{m} \frac{f_i(x_i + h e_\ell) - f_i(x_i - h e_\ell)}{2h} \, e_\ell.
\end{eqnarray}
Focus on the $\ell$-th component, we get
\begin{eqnarray}
&&\left|\left[\frac{A_i(t,x_i)}{r_i(t)}-\nabla f_i(t,x_i)\right]\right|_\ell\nonumber\\
&=&\left|\frac{f_i(t,x_i + h e_\ell) - f_i(t,x_i - h e_\ell)}{2h} -\partial_\ell f_i(t,x_i)\right|\nonumber\\
&=&\frac{1}{2h}\left|\int_{-h}^{h}[\partial_\ell f_i(t,x_i + u e_\ell)-\partial_\ell f_i(t,x_i)]du\right|\nonumber\\
&\leq&\frac{1}{2h}\int_{-h}^{h}\left|\partial_\ell f_i(t,x_i + u e_\ell)-\partial_\ell f_i(t,x_i)\right|du\nonumber\\
&\leq&\frac{1}{2h}\int_{-h}^{h}L_1|u|du\nonumber\\
&=&\frac{1}{2}L_1|h|=\frac{1}{2}L_1\gamma_tr_i.
\end{eqnarray}
Regarding $\frac{A_i(t,x_i) }{r_i(t)}-\nabla f_i(t,x_i)$ as $\theta_i(t)$, one has
\[
\|\theta_i(t)\|\leq C_A\gamma_t,
\]
where $C_A=\frac{1}{2}\sqrt{d}L_1(1+a)$. Moreover, as $\lim_{t\rightarrow\infty}\gamma_t=0$ and $0<\gamma_t\le 1$, we have $\lim_{t\rightarrow\infty}\|\theta_i(t)\|=0$. This completes the proof.
\end{proof}

\begin{remark}
The design of $A_i(t,x_i)$ does not aim for strict unbiasedness; rather, it incorporates a intentionally controlled bias. Under condition $\gamma_t\rightarrow0$, the estimator achieves asymptotic unbiasedness. Unbiased estimation necessitates $h\rightarrow0$, this requirement induces a catastrophic increase in noise variance \cite{2017-Ohad}, which affects numerical stabilization. Consequently, we adopt an appropriate perturbation parameter $h=\gamma_tr_i$, thereby sacrificing unbiasedness to ensure numerical stability and robustness.
\end{remark}

The next lemma gives a strong contraction property for the fast subsystem.
\begin{lemma}\label{lem:gi}
Consider the fast subsystem after freezing the slow variables $x_{i}, r_i$ at time $t_f$
\begin{align}
\varepsilon d g_{i} = \kappa ( A_i - r_i g_{i} ) d t+\beta_\varepsilon \sqrt{\varepsilon} \, \sigma_i^g \, dB^g(t),\,\,g_{i}(t_f)=g_i^{t_f},\label{freeze-gi}
\end{align}
where $A_i$ and $r_i$ are constants with $0<1-a<r_i<1+a$. Then, for any $ t \ge t_f $, there exist positive constants $ C_1,C_2 > 0 $ and $ c > 0 $, independent of $ \varepsilon $, such that
$$
\mathbb{E}\|g_{i}-g_i^*\|^2 \le C_1 e^{-c (t-t_f) / \varepsilon} + C_2 \beta_\varepsilon^2,
$$
and $\limsup_{t \to \infty} \mathbb{E}\|g_{i}-g_i^*\|^2 \le C_2 \beta_\varepsilon^2$, where $g_i^*=\frac{A_i}{r_i}$ is the steady state of \eqref{freeze-gi}.
In particular, when $\varepsilon \to 0$, based on Assumption \ref{A-sigma} we have
\[
\lim_{\varepsilon \to 0} \limsup_{t \to \infty} \mathbb{E}\|g_{i}-g_i^*\|^2 = 0.
\]
\end{lemma}
\begin{proof}
Defining the tracking error by $\delta^i_g=g_i-g_i^*$, one has
\[
d \delta^g_i = -\frac{\kappa r_i}{\varepsilon }\delta^g_i d t+\frac{\beta_\varepsilon}{\sqrt{\varepsilon}} \, \sigma_i^g \, dB^g(t).
\]
For a positive definite function $W_1=\delta^{gT}_i\delta^g_i$,  we have
\begin{eqnarray*}
\mathcal{L}W_1&=&2\delta^g_i\cdot\left(-\frac{\kappa r_i}{\varepsilon }\delta^g_i\right)+\frac{\beta_\varepsilon^2}{\varepsilon}trace[\sigma^T_g\sigma_g]\\
&\leq&-\frac{2\kappa r_i}{\varepsilon }\|\delta^g_i\|^2+\frac{\beta_\varepsilon^2}{\varepsilon}\|\sigma_g\|^2_F\\
&\leq&-\frac{2\kappa (1-a)}{\varepsilon }W_1+\frac{\beta_\varepsilon^2}{\varepsilon}\bar{\sigma}^2
\end{eqnarray*}
and so
\[
\frac{d}{dt}\mathbb{E}[W_1]\leq-\frac{2\kappa (1-a)}{\varepsilon }\mathbb{E}[W_1] +\frac{\beta_\varepsilon^2}{\varepsilon}\bar{\sigma}^2.
\]
Applying Gr\"{o}nwall's inequality, we obtain
\[
\mathbb{E}[W_1]\leq C_1e^{-c(t-t_f)/\varepsilon} +C_2\beta_\varepsilon^2,
\]
where $C_1=\mathbb{E}[W_1(t_0)]$, $C_2=\frac{\bar{\sigma}^2}{2\kappa (1-a)}$, and $c=2\kappa (1-a)$.
Consequently, we have $\limsup_{t \to \infty} \mathbb{E}\|\delta^g_i\|^2 \le C \beta_\varepsilon^2$ and
\[
\lim_{\varepsilon \to 0} \limsup_{t \to \infty} \mathbb{E}\|\delta^g_i\|^2 = 0.
\]
\end{proof}

\begin{remark}
Denote the instantaneous equilibrium set of the fast subsystem as
\[
\mathcal M_0(t) = \left\{ g_i^*(t) = \frac{A_i(t,x_i)}{r_i(t)}, \quad i=1,\ldots,N \right\}.
\]
Then for fixed values of $A_i$ and $r_i$, Lemma \ref{lem:gi} shows that the fast subsystem is mean-square exponentially stable with respect to its equilibrium point $g_i^*$, which means that the fast subsystem possesses a strong contraction property. Moreover, by Lemma \ref{lem:Ai}, one has
$g_i^*(t) =\nabla f_i(t,x_i)+\theta_i(t)$ with $\|\theta_i(t)\|\le C_A \gamma_t$.
Since the variables $x_i$, $A_i$, and $r_i$ evolve on a slower time scale, the instantaneous equilibrium set $\mathcal M_0(t)$ changes gradually. Consequently, the auxiliary variable $g_i$ can be regarded as a dynamic realization of the local gradient and serves as a gradient estimator for the slow subsystem. In view of Lemmas \ref{lem:Ai} and \ref{lem:gi}, its estimation accuracy is mainly determined by the finite-difference approximation error and the residual mean-square error associated with the fast subsystem.
\end{remark}

In order to describe the behaviour of SF-SMASs \eqref{slow}-\eqref{fast}, we construct the following reduced system by replacing $g_i$ with $\frac{A_i}{r_i}$ (denoted as R-SMASs):
\begin{align}
d \bar{x}_i =
 \bar{u}_i\, d t
+ \sigma_i^x \, d B_i^x(t), \, \, \bar{x}_{i}(0)=x_i^o\label{limit-xi}
\end{align}
with
\begin{eqnarray*}
\bar{u}_i&=&-\alpha_1 \frac{A_i}{r_i}-\alpha_2\sum_{j=1}^{n} a_{ij} (\bar{x}_i - \bar{x}_j)\nonumber\\
&&-\alpha_3\sum_{j=1}^{n} a_{ij} sig^p(\bar{x}_i - \bar{x}_j)-\alpha_4\sum_{j=1}^{n} a_{ij} sig^q(\bar{x}_i - \bar{x}_j),
\end{eqnarray*}
It should be noticed that such an R-SMASs is driven by the same Brownian motion and initial states as SF-SMASs \eqref{slow}-\eqref{fast}. Thus, the approximation error system $\delta_i^x = x_i - \bar{x}_i$ does not explicitly include Brownian motion.

Now we are in the position to give the following theorem which shows the approximation relation between SF-SMASs and R-SMASs.
\begin{theorem}\label{the:bar-x}
  Under Assumption \ref{A-sigma} and Lemma \ref{lem:gi}, there is a constant $ C(T) > 0 $, independent of $ \varepsilon $, such that
\[
\sup_{t \in [0,T]} \mathbb{E}\|x_i - \bar{x}_i\|^2 \le C(T)\left( \varepsilon + \beta_\varepsilon^2 \right), \, \forall i\in \mathcal{N}.
\]
Moreover,
$\lim_{\varepsilon \to 0} \sup_{0 \le t \le T} \mathbb{E}\|x_i - \bar{x}_i\|^2  = 0.$
\end{theorem}

\begin{proof}
Given the definition of the approximation error $\delta_i^x = x_i - \bar{x}_i$, one has
\[
d\delta_i^x = \left(u_i-\bar{u}_i\right)dt.
\]
For a positive definite function $W_2=\sum_{i=1}^{n}\delta^{xT}_i\delta^x_i$, we have
\begin{align*}
&\dot{W}_2\nonumber\\
=&2\sum_{i\in \mathcal{N}}\delta^{xT}_i\left(u_i-\bar{u}_i\right)\\
=&-2\alpha_1\sum_{i\in \mathcal{N}}\delta^{xT}_i\left(g_i-\frac{A_i}{r_i}\right) \nonumber\\
&-2\alpha_2\sum_{i,j\in \mathcal{N}} a_{ij} \delta^{xT}_i\left[(x_i - x_j)-(\bar{x}_i - \bar{x}_j)\right]\nonumber\\
&-2\alpha_3\sum_{i,j\in \mathcal{N}} a_{ij} \delta^{xT}_i \left[sig^p(x_i - x_j)-sig^p(\bar{x}_i - \bar{x}_j)\right] \nonumber\\
& -2\alpha_4\sum_{i,j\in \mathcal{N}} a_{ij} \delta^{xT}_i \left[sig^q(x_i - x_j)-sig^q(\bar{x}_i - \bar{x}_j)\right] \nonumber\\
=&-2\alpha_1\sum_{i\in \mathcal{N}}\delta^{xT}_i\delta^{g}_i \nonumber\\
&-\alpha_2\sum_{i,j\in \mathcal{N}} a_{ij} (\delta^{x}_i-\delta^{x}_j)^T\left[(x_i - x_j)-(\bar{x}_i - \bar{x}_j)\right]\nonumber\\
&-\alpha_3\sum_{i,j\in \mathcal{N}} a_{ij} (\delta^{x}_i-\delta^{x}_j)^T \left[sig^p(x_i - x_j)-sig^p(\bar{x}_i - \bar{x}_j)\right] \nonumber\\
& -\alpha_4\sum_{i,j\in \mathcal{N}} a_{ij} (\delta^{x}_i-\delta^{x}_j)^T \left[sig^q(x_i - x_j)-sig^q(\bar{x}_i - \bar{x}_j)\right]. \nonumber
\end{align*}
Based on the properties of the signum-type function given in Lemma \ref{lem:sig}, we can conclude that the last three terms are all non-positive and so
\begin{eqnarray*}
\dot{W}_2 &\leq&-2\alpha_1\sum_{i\in \mathcal{N}}\delta^{xT}_i\delta^{g}_i
\leq\alpha_1\sum_{i\in \mathcal{N}}\|\delta^x_i\|^2 +\alpha_1\sum_{i=1}^{n}\|\delta^g_i\|^2.
\end{eqnarray*}
Taking the expectation on both sides of the inequality yields
\begin{eqnarray*}
\frac{d}{dt}\mathbb{E}[W_2]&\leq&\alpha_1\mathbb{E}[W_2] +\alpha_1\sum_{i\in \mathcal{N}}\mathbb{E}\|\delta^i_g\|^2\\
&\leq&\alpha_1\mathbb{E}[W_2]+\alpha_1C_1ne^{-c t/\varepsilon}+\alpha_1C_2n\beta_{\varepsilon}^2.
\end{eqnarray*}
From Gr\"{o}nwall's inequality, we have
\begin{eqnarray*}
\mathbb{E}[W_2] &\leq& e^{\alpha_1 t}\mathbb{E}[W_2(0)]+\alpha_1C_1n\underbrace{\int_{0}^{t}e^{\alpha_1 (t-s)} e^{-cs/\varepsilon}ds}_{I_1}\\
&&+\alpha_1C_2n\underbrace{\int_{0}^{t}e^{\alpha_1 (t-s)}\beta_{\varepsilon}^2ds}_{I_2}.
\end{eqnarray*}
Clearly,
\begin{eqnarray*}
I_1&=&\int_{0}^{t}e^{\alpha_1 (t-s)} e^{-cs/\varepsilon}ds
=e^{\alpha_1 t}\int_{0}^{t}e^{-(\alpha_1+c\varepsilon) s} ds\\
&=&\frac{1}{\alpha_1+c/\varepsilon}\left(e^{\alpha_1 t}-e^{-ct/\varepsilon}\right),\\
I_2&=&\int_{0}^{t}e^{\alpha_1 (t-s)}\beta_{\varepsilon}^2ds
=e^{\alpha_1 t}\beta_{\varepsilon}^2\int_{0}^{t}e^{-\alpha_1s}ds\\
&=&\beta_{\varepsilon}^2\frac{e^{\alpha_1 t}-1}{\alpha_1}.
\end{eqnarray*}
Thus, based on the above discussions and the fact $x_i(0)=\bar{x}_i(0)$, there is a constant $C(T)$ such that
\[
\sup_{t \in [0,T]}\mathbb{E}\left[\sum_{i=1}^{n}\|x_i - \bar{x}_i\|^2 \right]\le C(T)\left( \varepsilon + \beta_\varepsilon^2 \right).
\]
Moreover, Assumption \ref{A-sigma} shows that
\[
\lim_{\varepsilon \to 0} \sup_{0 \le t \le T} \mathbb{E}\left[\sum_{i=1}^{n}\|x_i - \bar{x}_i\|^2 \right] = 0.
\]
By the non-negativity of the squared norm, these results naturally can be extended to each agent $i$.
\end{proof}
We would like to mention that the approximation error between SF-SMASs and the corresponding R-SMASs can be made arbitrarily small by selecting sufficiently small $\varepsilon$ and $\beta_\varepsilon$ via Theorem \ref{the:bar-x}, which means that the optimization and tracking performance of SF-SMASs \eqref{slow}-\eqref{fast} can be characterized through the analysis of R-SMASs \eqref{limit-xi}.

\subsection{Convergence Analysis of the Reduced System}
In this subsection, we will analyze the optimization and tracking performance of R-SMASs \eqref{limit-xi}, and then extend the conclusions to SF-SMASs  \eqref{slow}-\eqref{fast} using the approximate results established above.

\begin{theorem}\label{the:consensus}
Under Assumptions \ref{A-graph}-\ref{A-L1}, R-SMASs  \eqref{limit-xi} is Pfxc in probability and the stochastic settling time satisfies $t\leq \tau$ with
\[
\tau = \inf \left\{t:\mathbb{E}\left(\frac{1}{n}\sum_{i\in \mathcal{N}}\left\|\bar{x}_{i}-\frac{1}{n}\sum_{j\in \mathcal{N}}\bar{x}_{j}\right\|^2\right)\leq \frac{2\triangle}{n}\right\}
\]
 and
\[\mathbb{E}(\tau)\leq T =\frac{2(\bar{k}_2/\bar{k}_3)^{\frac{1-p}{q-p}}}{\bar{k}_2(1-p)}+ \frac{2(\bar{k}_2/\bar{k}_3)^{\frac{1-q}{q-p}}}{\bar{k}_3(q-1)},\]
where $k_1=2\alpha_2 \lambda_2[\mathcal{L}] -4\alpha_1 -2\alpha_1C_A>0$, $k_2=2^p\alpha_3 \lambda_2^{\frac{p+1}{2}}[\mathcal{L}_p]$, $k_3=2^q\alpha_4 \lambda_2^{\frac{q+1}{2}}[\mathcal{L}_q] d^{\frac{1-q}{2}}n^{1-q}$, $k_4=\frac{2\alpha_1L_3^2}{L_2} +n\alpha_1C_A+\bar{\sigma}^2$.
$\bar{k}_2=k_2-k_4/\triangle^p>0$, $\bar{k}_3=k_3-k_4/\triangle^q>0$,
$\triangle=\min \left\{\left(\frac{k_4}{(1-\omega)k_2}\right)^{\frac{2}{p+1}},
\left(\frac{k_4}{(1-\omega)k_3}\right)^{\frac{2}{q+1}}\right\}$, and $0<\omega<1$ is a design constant.
\end{theorem}
\begin{proof}
  Let $\tilde{x} =\frac{1}{n}\sum_{j\in \mathcal{N}}\bar{x}_{j}$. By Lemma \ref{lem:Ai}, we can construct the following average system:
\begin{align*}
d\tilde{x}
=&\frac{1}{n}\sum_{j\in \mathcal{N}} d\bar{x}_{j}
=-\frac{\alpha_1}{n}\sum_{j\in \mathcal{N}}\frac{A_j}{r_j}dt+\frac{1}{n}\sum_{j\in \mathcal{N}}\sigma^x_j dB^x(t)\\
=&-\frac{\alpha_1}{n}\sum_{j\in \mathcal{N}}\left(\nabla f_j(t,\bar{x}_j)+\theta_j(t)\right)dt+\frac{1}{n}\sum_{j\in \mathcal{N}}\sigma^x_j dB^x(t).
\end{align*}
Denote $e_i = \bar{x}_{i} - \tilde{x}$. Then it is easy to see that $\sum_{i\in \mathcal{N}}e_i=0$ and $ e_i-e_j = e_{ij} = \bar{x}_{i}-\bar{x}_{j}$. Thus, we can obtain the following error system:
\begin{align*}
de_i
=& d\bar{x}_{i} - d\tilde{x}\\
=& \left[-\alpha_1 (\nabla f_i(t,\bar{x}_i)+\theta_i(t))+\frac{\alpha_1}{n}\sum_{j\in \mathcal{N}}(\nabla f_j(t,\bar{x}_j)+\theta_j(t))\right.\\
&\left. - \alpha_2\sum_{j\in \mathcal{N}} a_{ij} (\bar{x}_i - \bar{x}_j)
-\alpha_3\sum_{j\in \mathcal{N}} a_{ij} sig^p(\bar{x}_i - \bar{x}_j)\right.\nonumber\\
&\left.-\alpha_4\sum_{j\in \mathcal{N}} a_{ij} sig^q(\bar{x}_i - \bar{x}_j)\right]dt
+\tilde{\sigma}^x_idB^x(t),
\end{align*}
where $\tilde{\sigma}^x_i=\sigma^x_i -\frac{1}{n}\sum_{j\in \mathcal{N}}\sigma^x_j$. For a positive definite function $V_1 = \frac{1}{2}\sum_{i\in \mathcal{N}}e_i^Te_i$, we have
\begin{eqnarray}
&&\mathcal{L}V_1\nonumber\\
&=&\sum_{i\in \mathcal{N}}e_i^T\left[-\alpha_1 \left(\nabla f_i(t,\bar{x}_i)+\theta_i(t)\right) \right.\nonumber\\
&&\left. +\frac{\alpha_1}{n}\sum_{j\in \mathcal{N}}\left(\nabla f_j(t,\bar{x}_j)+\theta_j(t)\right)\right.\nonumber\\
&&\left.- \alpha_2\sum_{j\in \mathcal{N}} a_{ij} (\bar{x}_i - \bar{x}_j)
-\alpha_3\sum_{j\in \mathcal{N}} a_{ij} sig^p(\bar{x}_i - \bar{x}_j)\right.\nonumber\\
&&\left.-\alpha_4\sum_{j\in \mathcal{N}} a_{ij} sig^q(\bar{x}_i - \bar{x}_j)\right]+\frac{1}{2}\sum_{i\in \mathcal{N}} trace[\tilde{\sigma}^{x T}_i\tilde{\sigma}^x_i]\nonumber\\
&=&\underbrace{-\frac{\alpha_1}{n}\sum_{i,j\in \mathcal{N}}e_i^T \left(\nabla f_i(t,\bar{x}_i)-\nabla f_j(t,\bar{x}_j)\right)}_{\Phi_1} \nonumber\\ &&\underbrace{-\frac{\alpha_1}{n}\sum_{i,j\in \mathcal{N}}e_i^T \left(\theta_i(t)-\theta_j(t)\right)}_{\Phi_2} \underbrace{-\alpha_2\sum_{i,j\in \mathcal{N}}a_{ij}e_{i}^Te_{ij}}_{\Phi_3}\nonumber\\
&&\underbrace{-\alpha_3\sum_{i,j\in \mathcal{N}}a_{ij}e_{i}^Tsig^p(e_{ij})}_{\Phi_4}\underbrace{-\alpha_4\sum_{i,j\in \mathcal{N}}a_{ij}e_{i}^Tsig^q(e_{ij})}_{\Phi_5}\nonumber\\
&&+\underbrace{\frac{1}{2}\sum_{i\in \mathcal{N}} trace[\tilde{\sigma}^{x T}_i\tilde{\sigma}^x_i]}_{\Phi_6}.\nonumber
\end{eqnarray}
Now we proceed to analyse each item in turn. Using Assumptions \ref{A-L1}, we get
\begin{eqnarray}
\Phi_1&=&-\frac{\alpha_1}{n}\sum_{i,j\in \mathcal{N}}e_i^T \left(\nabla f_i(t,\bar{x}_i)-\nabla f_j(t,\bar{x}_j)\right)\nonumber\\
&\leq&\frac{\alpha_1}{2n}\sum_{i,j\in \mathcal{N}}\|e_{ij}\|\cdot \left\|\nabla f_i(t,\bar{x}_i)-\nabla f_j(t,\bar{x}_j)\right\|\nonumber\\
&\leq&\frac{\alpha_1}{2n}\sum_{i,j\in \mathcal{N}}\|e_{ij}\| \left(L_2\|e_{ij}\|+L_3\right)\nonumber\\
&\leq&\frac{\alpha_1L_2}{2n}\sum_{i,j\in \mathcal{N}} (\|e_{ij}\|+\frac{L_3}{L_2})^2\nonumber\\
&\leq&\frac{\alpha_1L_2}{n}\sum_{i,j\in \mathcal{N}} \|e_{ij}\|^2+\frac{2\alpha_1L_3^2}{L_2}\nonumber\\
&\leq&2\alpha_1L_2\sum_{i=1}^{n} \|e_{i}\|^2+\frac{2\alpha_1L_3^2}{L_2}.\label{Phi1}
\end{eqnarray}
By Lemma \ref{lem:Ai}, we obtain
\begin{eqnarray}
\Phi_2&=&-\frac{\alpha_1}{2n}\sum_{i,j\in \mathcal{N}}e_{ij}^T \left(\theta_i(t)-\theta_j(t)\right)\nonumber\\
&\leq&\frac{\alpha_1}{2n}\sum_{i,j\in \mathcal{N}}\|e_{ij}\| \cdot \left\|\theta_i(t)-\theta_j(t)\right\|\nonumber\\
&\leq&\frac{\alpha_1 C_A \gamma_t}{n}\sum_{i,j\in \mathcal{N}}\|e_{ij}\|\nonumber\\
&\leq&\frac{\alpha_1 C_A \gamma_t}{2n}\sum_{i,j\in \mathcal{N}}\|e_{ij}\|^2+n\alpha_1 C_A \gamma_t\nonumber\\
&\leq&2\alpha_1 C_A \gamma_t \sum_{i\in \mathcal{N}}\|e_{i}\|^2\nonumber\\
&\leq&\alpha_1 C_A \sum_{i\in \mathcal{N}}\|e_{i}\|^2 +n\alpha_1 C_A. \label{Phi2}
\end{eqnarray}
Furthermore, it follows from Lemmas \ref{lem:inequality} and \ref{lem:laplacian} that
\begin{eqnarray}
\Phi_3&=&-\alpha_2 \sum_{i,j\in \mathcal{N}}a_{ij}e_{i}^Te_{ij}
= -\frac{\alpha_2}{2} \sum_{i,j\in \mathcal{N}}a_{ij}\|e_{ij}\|^2 \nonumber \\
&\leq&- \alpha_2\lambda_2[\mathcal{L}] V_1 \label{Phi3},\\
\Phi_4&=&-\alpha_3\sum_{i,j\in \mathcal{N}}a_{ij}e_{i}^Tsig^p(e_{ij})\nonumber \\
&=&-\frac{\alpha_3}{2}\sum_{i,j\in \mathcal{N}}a_{ij}\|e_{ij}\|_{p+1}^{p+1}
\leq-\frac{\alpha_3}{2}\sum_{i,j\in \mathcal{N}}a_{ij}\|e_{ij}\|^{p+1}\nonumber \\
&\leq&-\frac{\alpha_3}{2}\left(\sum_{i,j\in \mathcal{N}}a_{ij}^{\frac{2}{p+1}}\|e_{ij}\|^2\right)^{\frac{p+1}{2}}
\leq-\frac{\alpha_3}{2}(4\lambda_2[\mathcal{L}_p]V_1)^{\frac{p+1}{2}}\nonumber \\
&\leq&-2^p\alpha_3\lambda_2^{\frac{p+1}{2}}[\mathcal{L}_p]V_1^{\frac{p+1}{2}}\label{Phi4},\\
\Phi_5&=&-\alpha_4\sum_{i,j\in \mathcal{N}}a_{ij}e_{i}^Tsig^q(e_{ij})
=-\frac{\alpha_4}{2}\sum_{i,j\in \mathcal{N}}a_{ij}\|e_{ij}\|_{q+1}^{q+1}\nonumber\\
&\leq&-\frac{\alpha_4}{2}d^{\frac{1-q}{2}}\sum_{i,j\in \mathcal{N}}a_{ij}\|e_{ij}\|^{q+1}\nonumber\\
&\leq&-\frac{\alpha_4}{2}d^{\frac{1-q}{2}}n^{1-q}\left(\sum_{i,j=1}^{n}a_{ij}^{\frac{2}{q+1}}\|e_{ij}\|^2\right)^{\frac{q+1}{2}}\nonumber\\
&\leq&-\frac{\alpha_4}{2}d^{\frac{1-q}{2}}n^{1-q}(4\lambda_2[\mathcal{L}_q]V_1)^{\frac{q+1}{2}}\nonumber\\
&\leq&-2^q\alpha_4 \lambda_2^{\frac{q+1}{2}}[\mathcal{L}_q] d^{\frac{1-q}{2}}n^{1-q}V_1^{\frac{q+1}{2}}\label{Phi5}.
\end{eqnarray}
For the diffusion term, based on Assumption \ref{A-sigma}, we have
\begin{eqnarray}
\Phi_6&=&\frac{1}{2}\sum_{i\in \mathcal{N}} trace[\tilde{\sigma}_i^{xT}\tilde{\sigma}_i^x]
=\frac{1}{2}\sum_{i\in \mathcal{N}}\|\tilde{\sigma}_i^x\|_F^2\nonumber\\
&\leq& \frac{1}{2N^2}\sum_{i,j\in \mathcal{N}}\left\|\sigma_i^x -\sigma_j^x\right\|_F^2
 = \bar{\sigma}^2.\label{Phi6}
\end{eqnarray}
Thus, the results of \eqref{Phi1}-\eqref{Phi6} give
\begin{eqnarray*}
\mathcal{L}V_1&\leq&-k_1V_1-k_2V_1^{\frac{p+1}{2}}-k_3V_1^{\frac{q+1}{2}}+k_4\\
&\leq&-k_2V_1^{\frac{p+1}{2}}-k_3V_1^{\frac{q+1}{2}}+k_4,
\end{eqnarray*}
where $k_1=2\alpha_2\lambda_2[\mathcal{L}]-4\alpha_1-2\alpha_1C_A>0$, $k_2=2^p\alpha_3\lambda_2^{\frac{p+1}{2}}[\mathcal{L}_p]$, $k_3=2^q\alpha_4\lambda_2^{\frac{q+1}{2}}[\mathcal{L}_q]d^{\frac{1-q}{2}}n^{1-q}$, $k_4=\frac{2\alpha_1L_3^2}{L_2}+n\alpha_1C_A+\bar{\sigma}^2$.

Finally, it follows from Lemma \ref{lem:Pfxs} that R-SMASs \eqref{limit-xi} is Pfxc in probability and the stochastic settling time satisfies $t\leq \tau$ with
\[
\tau = \inf \left\{t:\mathbb{E}\left(\frac{1}{n}\sum_{i\in \mathcal{N}}\left\|\bar{x}_{it}-\frac{1}{n}\sum_{j\in \mathcal{N}}\bar{x}_{jt}\right\|^2\right)\leq \frac{2\triangle}{n}\right\}
\]
 and
\[\mathbb{E}(\tau)\leq T =\frac{2(\bar{k}_2/\bar{k}_3)^{\frac{1-p}{q-p}}}{\bar{k}_2(1-p)}+ \frac{2(\bar{k}_2/\bar{k}_3)^{\frac{1-q}{q-p}}}{\bar{k}_3(q-1)},\]
where $\bar{k}_2=k_2-k_4/\triangle^p>0$, $\bar{k}_3=k_3-k_4/\triangle^q>0$,
$\triangle=\min \left\{\left(\frac{k_4}{(1-\omega)k_2}\right)^{\frac{2}{p+1}},
\left(\frac{k_4}{(1-\omega)k_3}\right)^{\frac{2}{q+1}}\right\}$, and $0<\omega<1$ is a design constant.
\end{proof}

The above theorem guarantees that R-SMASs \eqref{limit-xi} can achieve biased consensus within a fixed time.
Below, we analyse the tracking performance of R-SMASs \eqref{limit-xi} with respect to $x_t^*$.

\begin{theorem}\label{the:optimization}
Under Assumptions \ref{A-graph}-\ref{A-x*}, the tracking errors of R-SMASs \eqref{limit-xi} with respect to the optimal trajectory $x^*_t$ is MS-GEUB, i.e.
\begin{align*}
&\mathbb{E}\left(\frac{1}{n}\sum_{i\in \mathcal{N}}\|\bar{x}_{i}-x_t^*\|^2\right)\\
\leq& \mathbb{E}\left(\frac{1}{n}\sum_{i\in \mathcal{N}}\|\bar{x}_{i}^o-x_0^*\|^2\right)e^{-k_5t} +\frac{2k_6}{nk_5}(1-e^{-k_5t})
\end{align*}
  for all $t>0$, where $k_5=\alpha_1\mu-\alpha_1 +\alpha_2\lambda_2[\mathcal{L}]-1>0$ and $k_6=\frac{n}{2}(\alpha_1C_A+M+\bar{\sigma}^2)$.
\end{theorem}

\begin{proof}
Denote $e^*_i = \bar{x}_{i} - x^*_t$ and $ e^*_i-e^*_j = e^*_{ij} = \bar{x}_{i}-\bar{x}_{j}$. For a positive definite function $V_2 = \frac{1}{2}\sum_{i\in \mathcal{N}}e^{*T}_ie^*_i$, it follows from Lemma \ref{lem:Ai} and Assumption \ref{A-sigma} that
\begin{eqnarray}
&&\mathcal{L}V_2\nonumber\\
&\leq& \underbrace{-\alpha_1\sum_{i\in \mathcal{N}}e_i^{*T} \nabla f_i(t,\bar{x}_i)}_{\Phi_7} \underbrace{-\alpha_1\sum_{i\in \mathcal{N}}e_i^{*T} \theta_i(t)}_{\Phi_8}
\underbrace{-\sum_{i\in \mathcal{N}}e_i^{*T} \dot{x}_t^*}_{\Phi_{9}}\nonumber\\
&&\underbrace{- \alpha_2\sum_{i,j\in \mathcal{N}} a_{ij} e_{i}^{*T}e_{ij}^{*} }_{\Phi_{10}}
\underbrace{-\alpha_3\sum_{i,j\in \mathcal{N}}a_{ij}e^{*T}_{i}sig^p(e^*_{ij})}_{\Phi_{11}}\nonumber\\
&&\underbrace{-\alpha_4\sum_{i,j\in \mathcal{N}}a_{ij}e^{*T}_{i}sig^q(e^*_{ij})}_{\Phi_{12}}
+\frac{n\bar{\sigma}^2}{2}.
\end{eqnarray}
Based on Assumption \ref{A-L1}, we have
\begin{align}
\Phi_7=&-\alpha_1\sum_{i\in \mathcal{N}}e_i^{*T} \nabla f_i(t,\bar{x}_i)\nonumber\\
\leq&-\alpha_1\sum_{i\in \mathcal{N}}e_i^{*T} \nabla f_i(t,\bar{x}_i)\nonumber\\
&+\alpha_1\left(\sum_{i\in \mathcal{N}}f_i(t,\bar{x}_i)-\sum_{i\in \mathcal{N}}f_i(t,x_t^*)\right)\nonumber\\
=&-\alpha_1\sum_{i\in \mathcal{N}}\left[ (\bar{x}_i-x_t^*)\nabla f_i(t,\bar{x}_i) -f_i(t,\bar{x}_i) +f_i(t,x_t^*)\right]\nonumber\\
\leq&-\frac{\alpha_1\mu}{2}\sum_{i\in \mathcal{N}} \|e_i^{*}\|^2.\label{Phi7}
\end{align}
According to Lemma \ref{lem:Ai} and Assumption \ref{A-sigma}, one has
\begin{eqnarray}
\Phi_8&=&-\alpha_1\sum_{i\in \mathcal{N}}e_i^{*T} \theta_i(t)
\leq\alpha_1\sum_{i\in \mathcal{N}}\|e_i^{*}\| \cdot \left\|\theta_i(t)\right\|\nonumber\\
&\leq&\frac{\alpha_1}{2} \sum_{i\in \mathcal{N}}\left(\|e_i^{*}\|^2 + \left\|\theta_i(t)\right\|^2\right)\nonumber\\
&\leq&\frac{\alpha_1}{2} \sum_{i\in \mathcal{N}} \|e_i^{*}\|^2+\frac{n\alpha_1 C_A }{2} \label{Phi8}\\
\Phi_{9}&\leq&\sum_{i\in \mathcal{N}}\|e_i^{*}\| \cdot\|\dot{x}_t^*\|
\leq\frac{1}{2}\sum_{i\in \mathcal{N}}\|e_i^{*}\|^2 +\frac{nM}{2}.\label{Phi9}
\end{eqnarray}
As for $\Phi_{10}$, $\Phi_{11}$ and $\Phi_{12}$, we have
\begin{eqnarray}
&&\Phi_{10}+\Phi_{11}+\Phi_{12}\nonumber\\
&\leq&-\frac{\alpha_2}{2}\sum_{i,j\in \mathcal{N}}a_{ij}\|e^*_{ij}\|^2
-\alpha_3\sum_{i,j\in \mathcal{N}}a_{ij} e^{*T}_{i}sig^p(e^*_{ij}) \nonumber\\
&&-\alpha_4\sum_{i,j\in \mathcal{N}}a_{ij}e^{*T}_{i}sig^q(e^*_{ij})\nonumber\\
&\leq&-\frac{\alpha_2 \lambda_2[\mathcal{L}]}{2}\sum_{i\in \mathcal{N}}\|e^*_{i}\|^2
-\frac{\alpha_3}{2}\sum_{i,j\in \mathcal{N}}a_{ij}\|e^{*}_{ij}\|_{p+1}^{p+1}\nonumber\\
&&-\frac{\alpha_4}{2}\sum_{i,j\in \mathcal{N}}a_{ij}\|e^{*}_{ij}\|_{q+1}^{q+1}\nonumber\\
&\leq&-\frac{\alpha_2 \lambda_2[\mathcal{L}]}{2}\sum_{i\in \mathcal{N}}\|e^*_{i}\|^2\label{Phi101112}.
\end{eqnarray}

Combining the results of \eqref{Phi7}-\eqref{Phi101112} gives
\[
\mathcal{L}V_2\leq-k_5V_2+k_6,
\]
where $k_5=\alpha_1\mu-\alpha_1 +\alpha_2\lambda_2[\mathcal{L}]-1>0$, $k_6=\frac{n}{2}(\alpha_1C_A+M+\bar{\sigma}^2)$, i.e.
\begin{eqnarray}
&&\mathbb{E}\left[\frac{1}{n}\sum_{i\in \mathcal{N}}\|\bar{x}_{i}-x_t^*\|^2\right]\nonumber\\
&\leq& \mathbb{E}\left[\frac{1}{n}\sum_{i\in \mathcal{N}}\|\bar{x}_{i}^o-x_0^*\|^2\right]e^{-k_5t} +\frac{2k_6}{nk_5}(1-e^{-k_5t})\nonumber
\end{eqnarray}
and $\limsup_{t\rightarrow\infty}\mathbb{E} \left[\frac{1}{n}\sum_{i\in \mathcal{N}}\|\bar{x}_{i}-x_t^*\|^2\right] \leq\frac{2k_6}{nk_5}$.
\end{proof}

It is worth mentioning that SF-SMASs inherits the consensus and optimization properties of R-SMASs within an arbitrarily small approximation error via Theorems \ref{the:bar-x}-\ref{the:optimization}. In particular, each agent can achieve Pfxc and MS-GEUB relative to the time-varying optimal trajectory. Furthermore, by selecting sufficiently small values of $\varepsilon$ and $\beta_\varepsilon$, the corresponding performance bounds can be made arbitrarily close to those of  R-SMASs.

\section{Simulations}
This section will verify the above conclusions through numerical simulations. Consider a network of 10 agents aims to track a globally optimal trajectory using only local information. Here, each agent has access only to its own objective function value and the state information of its neighbors, and they cooperatively solve the distributed optimization problem.

The initial state of each agent is uniformly randomly selected from $[-3,3]^2$. Assume that Graph $\mathcal{G}$ is a static, undirected, cyclic graph. The time-varying objective function for each agent $i$ is $f_i(t,x)=\frac{1}{2}\|x-c_i\sin(t)\|^2$, where $c_1=[-2,1]^T$, $c_2=[-1.5,2]^T$, $c_3=[-1,1.5]^T$, $c_4=[-0.5,1]^T$, $c_5=[2,0.5]^T$, $c_6=[1,-0.5]^T$, $c_7=[0.5,-1]^T$, $c_8=[1,-1.5]^T$, $c_9=[1.5,-2]^T$, $c_{10}=[2,2]^T$. It can be computed that the optimal trajectory $x_{t}^*=[0.3\sin(t),0.3\sin(t)]^T$.

Thus, the SF-SMASs considered in this section can be described as follows:
\begin{eqnarray}
d x_i &=& u_i\, d t
+\sigma^x_i \, d B_i^x(t) \label{sim-xi}\\
\varepsilon d g_i &=& \left[ A_i(t,x_i) - r_i(t) g_i \right] d t+ \beta_{\varepsilon}\sqrt{\varepsilon}\sigma^g_i \, d B^g(t),\label{sim-gi}
\end{eqnarray}
where
\begin{eqnarray*}
u_i&=&-3g_i -25\sum_{j=1}^{n} a_{ij} (x_i - x_j)\nonumber\\
&&-20\sum_{j=1}^{n} a_{ij} sig^{0.6}(x_i - x_j)\nonumber\\
&&-10\sum_{j=1}^{n} a_{ij} sig^{1.4}(x_i - x_j),\\
\sigma_i^x&=&\begin{pmatrix}
 0.5\sin(\pi t)  & 0 \\
 0 &  0.5\cos(\pi t)
\end{pmatrix},\\
\sigma_i^g&=&\begin{pmatrix}
 0.5\sin(\pi t)  & 0 \\
 0 &  0.5\cos(\pi t)
\end{pmatrix},\\
\gamma_t&=&\frac{1}{(t+1)^{0.25}},\\
r_i(t) &=& r_i^k, \quad t \in [k T_r, (k+1) T_r)\,\, \mbox{and} \,\, T_r \gg \varepsilon
\end{eqnarray*}
with $r_i^k \sim \mbox{Unif} [0.9,\, 1.1]$, $\varepsilon=0.01$, and $T_r=1$.

It is easy to see that all the conditions can be satisfied in Theorems \ref{the:consensus}-\ref{the:optimization}. Furthermore, we can calculate that the global average consensus error is Pfxc in probability and the stochastic settling time
\[
\tau = \inf \left\{t:\mathbb{E}\left(\frac{1}{n}\sum_{i\in \mathcal{N}}\left\|\bar{x}_{i}-\frac{1}{n}\sum_{j\in \mathcal{N}}\bar{x}_{j}\right\|^2\right)\leq 18.4\right\}
\]
 and
$\mathbb{E}(\tau)\leq 22.63s$.
Moreover, The tracking error of $x_t^*$ satisfies
$\limsup_{t\rightarrow\infty}\mathbb{E} \left[\frac{1}{n}\sum_{i\in \mathcal{N}}\|\bar{x}_{i}-x_t^*\|^2\right] \leq0.1806$.

 \begin{figure}[!t]
  \centering
    \includegraphics[width=0.45\textwidth]{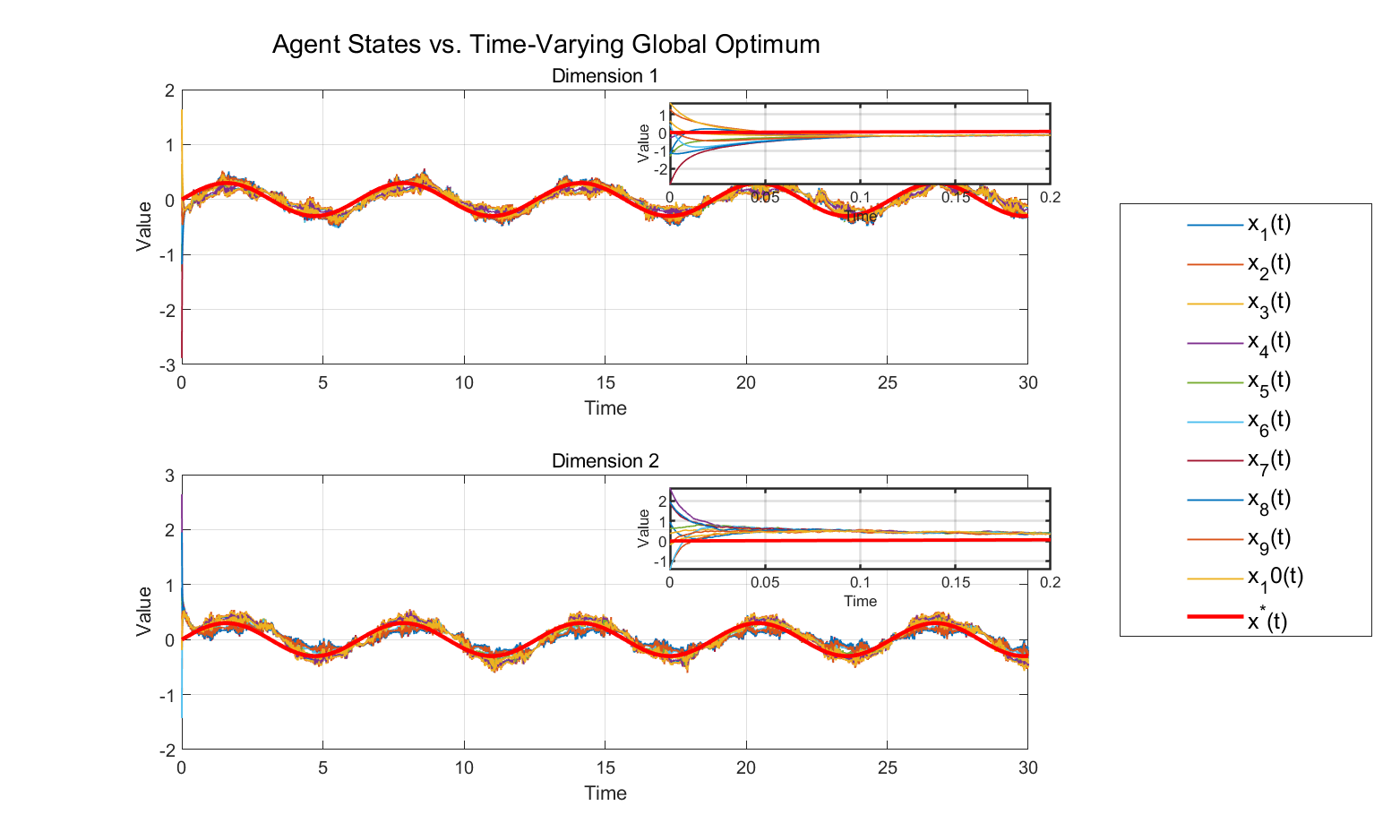}
    \caption{State $x_i$ under system \eqref{sim-xi} }
    \label{fig:x}
\end{figure}
 \begin{figure}[!t]
  \centering
    \includegraphics[width=0.45\textwidth]{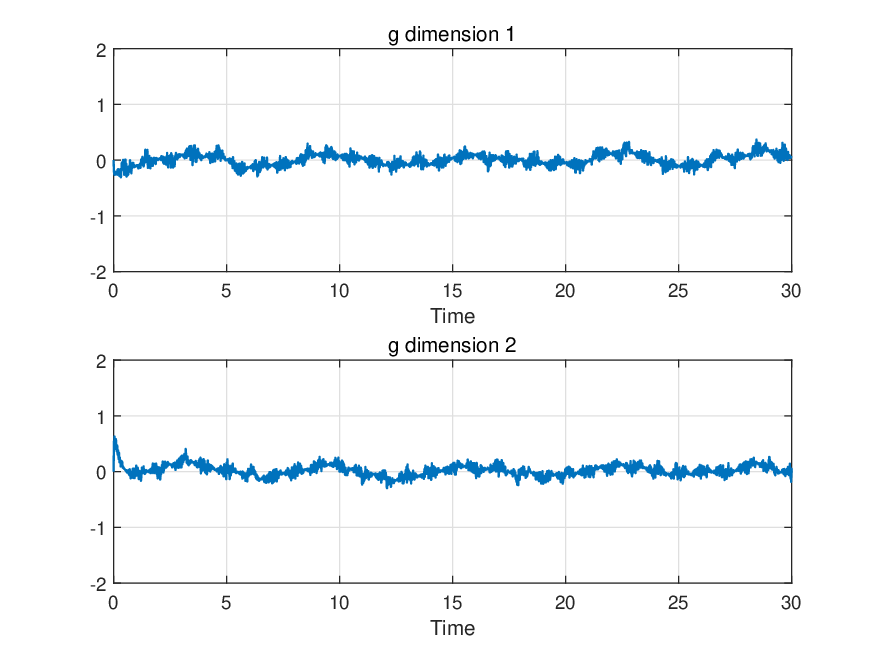}
    \caption{The global smoothed gradient estimates calculated based on $g_i$}
    \label{fig:g}
\end{figure}
 \begin{figure}[!t]
  \centering
    \includegraphics[width=0.45\textwidth]{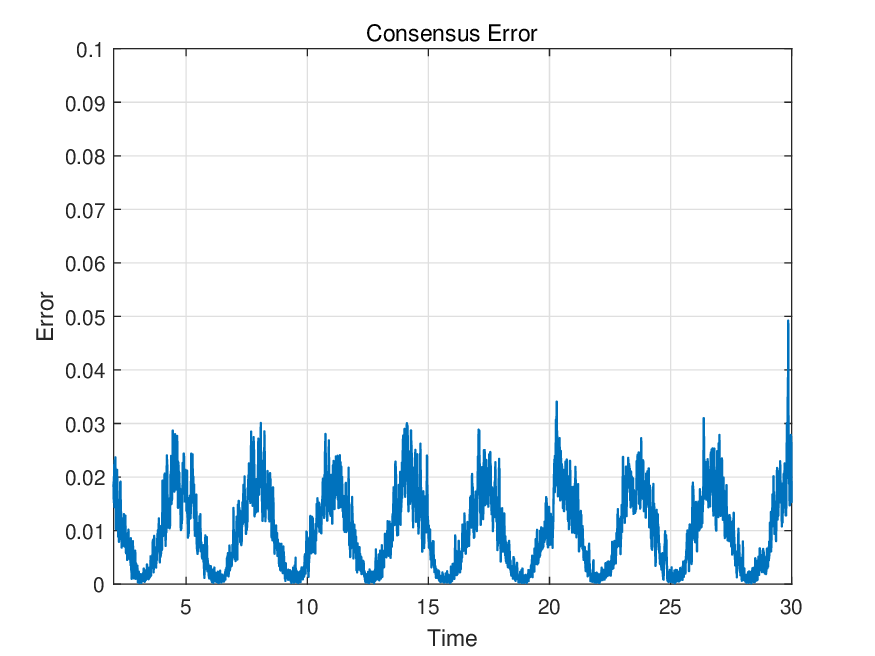}
    \caption{The global average consensus error $\frac{1}{n}\sum_{i\in \mathcal{N}}\left\|x_{i}-\frac{1}{n}\sum_{j\in \mathcal{N}}x_{j}\right\|^2$}
    \label{fig:consensus_error}
\end{figure}
 \begin{figure}[!t]
  \centering
    \includegraphics[width=0.45\textwidth]{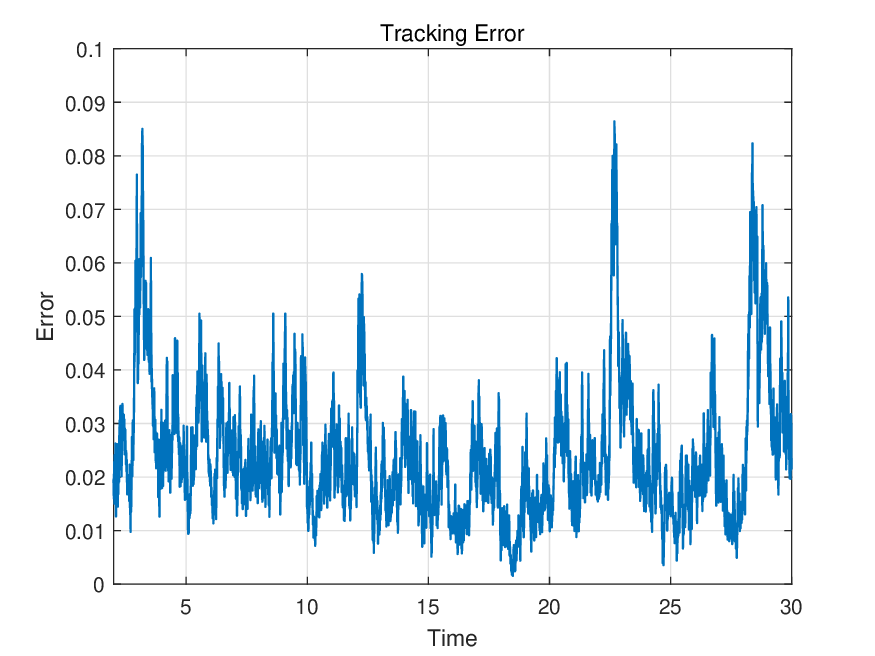}
    \caption{The global tracking error $\frac{1}{n}\sum_{i\in \mathcal{N}}\|\bar{x}_{i}-x_t^*\|^2$}
    \label{fig:tracking_error}
\end{figure}
The results are displayed in the Figure \ref{fig:x}-\ref{fig:tracking_error}.
Observing Figures \ref{fig:x}, we can obtain that all agents have rapidly achieved Pfxc in probability,
where the red curves represent the optimal trajectory in Figure \ref{fig:x}.
Figure \ref{fig:g} shows that the global smooth gradient estimates obtained by \eqref{sim-gi} continue to fluctuate around zero.
Figure \ref{fig:consensus_error} shows the global average consensus error. As can be seen, the error is much smaller than the upper bound we calculated, indicating that the upper bound provided by the Theorem \ref{the:consensus} is a relatively relaxed bound.
Figure \ref{fig:tracking_error} shows the global tracking error. As can be seen, this error is smaller than the upper bound we calculated, which validates the result of the Theorem \ref{the:optimization}.
As stated above, the proposed system \eqref{sim-xi}-\eqref{sim-gi} successfully achieves all design objectives.
\section{Conclusion}
This paper proposes a slow-fast stochastic framework for zero-order distributed time-varying optimization. The design integrates distributed optimization, gradient-free estimation, and stochastic singular perturbation techniques into a unified continuous-time framework. By applying stochastic singular perturbation techniques and stochastic Lyapunov theory, rigorous convergence and tracking results are established, providing a theoretical foundation for zero-order distributed time-varying optimization of SMASs.

Although the developed framework provides a systematic approach for distributed time-varying optimization problems, several challenges remain to be addressed. An important future research area is the development of adaptive time-scale mechanisms, enabling them to automatically adjust parameters in response to environmental changes and optimization requirements. Another interesting research area is the study of non-convex time-varying optimization problems, where the interaction between stochastic gradient estimation and moving equilibrium points becomes significantly more complex. Furthermore, extending the proposed methodology to learning-based optimization architectures may further broaden its scope of application. Finally, establishing a more refined characterization of tracking errors is also a highly promising area for future theoretical research.

\bibliographystyle{IEEEtran}
\bibliography{reference}

\end{document}